\documentclass[10pt]{amsart}
\usepackage{amssymb,amsthm,amsmath,amsfonts}
\usepackage[numbers]{natbib}
\usepackage{hyperref}
\usepackage{enumerate}

\usepackage{xcolor}

\usepackage[normalem]{ulem}

\makeatletter
\g@addto@macro\th@plain{\thm@headpunct{}}
\makeatother

\newtheorem{theorem}{Theorem}[section]
\newtheorem{proposition}[theorem]{Proposition}
\newtheorem{lemma}[theorem]{Lemma}
\newtheorem{remark}[theorem]{Remark}

\newcommand{\xx}{ {\textbf x} }
\newcommand{\ab}{ {\textbf a} }
\newcommand{\bb}{ {\textbf b} }
\newcommand{\cb}{ {\textbf c} }

\newcommand{\yy}{ {\textbf y} }
\renewcommand{\ss}{ {\textbf s} }
\newcommand{\ttt}{ {\textbf t} }

\newcommand{\ub}{ {\textbf u} }
\newcommand{\vb}{ {\textbf v} }

\newcommand{\VV}{ \mathcal{S}_+}

\newcommand{\DD}{\mathcal{D}}

\newcommand{\II}{1{\hskip -2.45 pt}\hbox{l}}
\newcommand{\I}{\mathrm{I}}
\newcommand{\R}{\mathbb{R}}
\newcommand{\tr}{\mathrm{tr}\,}

\providecommand{\scalar}[1]{\left\langle#1\right\rangle}

\title{A Matsumoto-–Yor characterization for Kummer and Wishart random matrices}
\author[B. Ko\l{}odziejek]{Bartosz Ko\l{}odziejek}
\address{Faculty of Mathematics and Information Science\\Warsaw University of Technology\\Koszykowa 75\\00-662 Warsaw, Poland}
\email{b.kolodziejek@mini.pw.edu.pl}



\begin{document}

\begin{abstract}
In the paper we resolve positively the conjecture on a characterization of matrix Kummer and Wishart laws through independence property, which was posed in [Koudou, Statist. Probab. Lett. 82 (2012), 1903--1907]. Apart from the probabilistic result, we determine the general solution of the functional equation associated to the characterization problem under weak assumptions.  
\end{abstract}

\maketitle

\section{Introduction}
Bernstein \cite{B41} proved, under some technical assumptions, that if random variables $X$ and $Y$ are independent, then random variables $U=X+Y$ and $V=X - Y$ are independent if and only if $X$ and $Y$ are Gaussian. 
Under additional assumptions that $X$ and $Y$ have densities, this result is equivalent to finding the general solution of the following functional equation
$$f_X(x)f_Y(y)=2f_{U}(x+y)f_V(x-y),\qquad \mbox{a.e. }(x,y)\in\R^2,$$
where $f_X$, $f_Y$, $f_U$ and $f_V$ are unknown densities (thus measurable and a.e. non-negative) of respective random variables.
Under mild regularity conditions, this functional equation has a solution
\begin{align*}
f_X(x)&=\exp\{A x^2+B_1 x+C_1\}, 
 &f_U(x)=\exp\{\tfrac12A x^2+\tfrac12(B_1+B_2) x+C_3\}, \\
f_Y(x)&=\exp\{A x^2+B_2 x+C_2\},
 &f_V(x)=\exp\{\tfrac12A x^2+\tfrac12(B_1-B_2) x+C_4\},
\end{align*}
for a.e. $x\in\R$, for some $A, B_i, C_i\in\R$ with $C_1+C_2=C_3+C_4+\log(2)$. Since $f_X$, $f_Y$, $f_U$ and $f_V$ have to be integrable on $\R$, $A$ has to be strictly negative so that $X$ and $Y$ are Gaussian with the same variance.

Many similar examples of the so-called independence characterizations have been identified throughout the years. They follow the following general scheme: determine the distributions of $X$ and $Y$ if they are independent and the components of $(U,V)=\psi(X,Y)$ are independent, where $\psi$ is some given {function} defined on the support of $(X,Y)$. If $\psi$ is a diffeomorphism, under the assumption of existence of densities of $X$ and $Y$, such problem is equivalent to solving the following associated functional equation
\begin{align}\label{eq007}
f_X(x)f_Y(y)=J(x,y) f_U(\psi_1(x,y))f_V(\psi_2(x,y)),\qquad \mbox{a.e. }(x,y)\in\R^2,
\end{align}
where $J=\left|\frac{\mathrm{d}\psi_1}{\mathrm{d}x}\frac{\mathrm{d}\psi_2}{\mathrm{d}y}-\frac{\mathrm{d}\psi_1}{\mathrm{d}y}\frac{\mathrm{d}\psi_2}{\mathrm{d}x}\right|$ is the Jacobian of $\psi$. Obviously, not every diffeomorphism yields a probabilistic solution (in which we are here particularly interested) and it is still not clear how to choose appropriate $\psi$.

There is a very interesting family of $\psi$'s which {was} introduced in \cite{KV12}. From the probabilistic point of view, this family is somehow related to the so-called Matsumoto-Yor property. Koudou and Vallois in \cite{KV12} (see also \cite{KV11}) considered $\psi^{(f)}(x,y)=\left(f(x+y),f(x)-f(x+y)\right)$ where $f$ is some regular function and asked the following question: for which $f$ there exist {independent} $X$ and $Y$ such that $U$ and $V$ are independent. The classical Matsumoto-Yor property (see \cite{JW2002MY,Wes02}) is obtained for $f^{(1)}(x)=x^{-1}$. If $f^{(2)}(x)=\log(x)$ we obtain the so-called Lukacs property (see \cite{Lukacs1955, Baker1976, Lajko1979, Mesz2010, WES4}). 
Another important case identified in \cite{KV12} was $f^{(3)}(x) = \log(1 + x^{-1})$, which is the subject of present study.

All above-mentioned cases have their matrix-variate analogues, where $X$ and $Y$ are considered to be random matrices (or equivalently, $X$ and $Y$ to be random variables valued in the set of matrices). Such properties are usually much harder to prove due to additional noncommutativity of matrix multiplication, which is not witnessed if $X$ and $Y$ are $1$-dimensional. 
A generalization of the classical Matsumoto-Yor property was considered in {\cite{Wes02, WesLe00, BK17}} along with a characterization of probability laws having this property. This characterization was proven via solving the associated functional equation
\begin{align}\label{eqluk}
A(\xx)+B(\yy)=C(\xx+\yy)+D\left(\xx^{-1}-(\xx+\yy)^{-1}\right),\qquad (\xx,\yy)\in\VV^2,
\end{align}
where $\VV$ is the cone of positive definite real matrices of {full} rank and $A, B, C, D\colon\VV\to\R$ are unknown functions. This functional equation is a matrix-variate version of \eqref{eq007} after taking the logarithm of both sides, which is allowed if one assumes that $f_X$ and $f_Y$ are positive on $\VV$. Eq. \eqref{eqluk} was solved in \cite{BK17} under the assumption that $A$ and $B$ are continuous, which is the best result available at the moment. Actually, \eqref{eqluk} was considered there in a more general setting, that is, on symmetric cones of which $\VV$ is the prime example. 

A generalization of other properties identified in \cite{KV12} to matrices is not automatic and this is due to the fact there is no natural notion of division of matrices. For example, Lukacs property on $\VV$ is stated with $\psi^{(2)}(\xx,\yy)=\left(\xx+\yy,g(\xx+\yy)\cdot\xx\cdot g(\xx+\yy)^\top\right)$, where $g\colon \VV\mapsto M_r$ is such that $g(\xx)\cdot\xx\cdot g(\xx)^\top=\I$ for any $\xx\in\VV$. Here ``$\cdot$'' denotes the ordinary matrix multiplication, $\I$ is the identity matrix in $\VV$ and $\xx^\top$ denotes the transpose of $\xx$. One can take for example $g(\xx)=(\xx^{1/2})^{-1}$, where $\xx^{1/2}$ is the unique positive definite square root of $\xx=\xx^{1/2}\cdot\xx^{1/2}$. A characterization of laws having Lukacs property was considered in \cite{BW2002, BK13,BK16a}, and in the latter paper the general solution of the following associated functional equation was found
\begin{align*}
A(\xx)+B(\yy)=C(\xx+\yy)+D\left(g(\xx+\yy)\cdot\xx\cdot g(\xx+\yy)^\top\right),\qquad (\xx,\yy)\in\VV^2
\end{align*}
under the assumption that $A$ and $B$ are continuous.

In the present paper we will consider a generalization of $f^{(3)}$ to $\VV$, which was introduced in \cite{Kou12}:
$$\psi^{(3)}(\xx,\yy)=\left(\xx+\yy, (\I+(\xx+\yy)^{-1})^{1/2}\cdot(\I+\xx^{-1})^{-1}\cdot (\I+(\xx+\yy)^{-1})^{1/2}\right).$$
In \cite{Kou12} it was shown that if $X$ and $Y$ are independent, $X$ has matrix Kummer law and $Y$ has Wishart law with suitable parameters, then the components of $(U,V)=\psi^{(3)}(U,V)$ are also independent. In the very same paper, a conjecture is posed that this independence property characterizes matrix-Kummer and Wishart laws. 

There are two main result of the present paper. One is to find the general solution of the functional equation associated with the characterization conjectured in \cite{Kou12} and the second is to prove this conjecture. 
In order to solve the functional equation, we will use techniques developed in \cite{BK16a} and further used in \cite{BK16b, BK16c, BK17, AP17}. There are other independence characterizations involving Kummer distribution, which use different methods \cite{JW15, PW17}; see also a characterization of vector-variate Kummer law in \cite{PW16}. Finally, it is important to note that in some sense one can pass with the rank $r$ of $\VV$ to infinity and obtain properties and characterizations of laws of free random variables; see \cite{KSZ15} for Lukacs and \cite{KSZ17} for Matsumoto-Yor property in free probability.

The paper is organized as follows. In the next Section we set the notation and recall some basic properties of the determinant. Section 3 is devoted to study the most important technical step, that is, solving the functional equation (Theorem \ref{mainfun}) related to the characterization of probability laws. In Section 4 we introduce the probability laws on $\VV$ and prove the second main result (Theorem \ref{mainprob}).

\section{Notation and preliminaries}
Let $\mathrm{Sym}(r,\R)$ denote the set of {real} symmetric matrices of size $r\times r$. We endow the space $\mathrm{Sym}(r,\R)$ with the scalar product $\scalar{\xx,\yy}=\tr(\xx\cdot\yy)$. In $\mathrm{Sym}(r,\R)$ we consider the cone $\VV$ of positive definite symmetric matrices of rank $r$. Elements of $\mathrm{Sym}(r,\R)$, if non-random, will be denoted by bold letters. It should be noted that results of the present paper remain true for all symmetric cones, but we stick to $\VV$ so that our arguments are easier accessible to a wider audience.

For $\xx,\yy\in \VV$ we define a multiplication 
$$\xx\circ\yy=\xx^{1/2}\cdot \yy\cdot \xx^{1/2},$$
where $\xx^{1/2}$ is the unique positive definite square root of $\xx=\xx^{1/2}\cdot \xx^{1/2}$. The product $\circ$ is inner {(that is, $\xx\circ\yy\in\VV$ if $\xx,\yy\in\VV$)}, but {neither} commutative nor associative. The identity matrix $\I$ is the neutral element for $\circ$. If $\xx$ and $\yy$ commute, then $\xx\circ\yy=\xx\cdot\yy$.

Let $\det$ denote the usual determinant in $\VV$. We have
\begin{align}\label{det0}
\det(\xx\circ\yy)=\det(\xx) \det(\yy),\qquad (\xx,\yy)\in\VV^2
\end{align}
and it will be crucial for us that this property actually characterizes determinant (see \cite[Theorem 3.4]{BK15}). We will use this fact {several times} throughout the paper.


\begin{proposition}\label{ProK1}
	For $\xx,\yy\in\VV$ set $\ub=\left(\I+(\xx+\yy)^{-1}\right)\circ\left(\I+\xx^{-1}\right)^{-1}$.
	Then $\ub\in\VV$, $\I-\ub\in\VV$ and
	\begin{align*}
		\det(\ub)&=\frac{\det(\I+\xx+\yy)}{\det(\xx+\yy)}\frac{\det(\xx)}{\det(\I+\xx)},\\
		\det(\I-\ub)&=\frac{\det(\yy)}{\det(\I+\xx)\det(\xx+\yy)}
	\end{align*}
\end{proposition}
\begin{proof}
	Use \eqref{det0} and
	$$\det(\ab^{-1}-\bb^{-1})=\det\left(\ab^{-1}\cdot(\bb-\ab)\cdot\bb^{-1}\right)=\frac{\det(\bb-\ab)}{\det(\ab) \det(\bb)}$$
	for nonsingular $\ab$ and $\bb$.
\end{proof}

\section{Functional equations}
By $(\I+\VV)$ we denote the set $\{\I+\xx\colon \xx\in\VV\}$. We will need the following {result} later on. 
\begin{lemma}\label{PEX}
	Let $A,B,C\colon\I+\VV\to\R$ be continuous. Assume that
	$$A(\xx)+B(\yy)=C(\xx\circ\yy),\qquad (\xx,\yy)\in(\I+\VV)^2.$$
	Then there exist real constants $p, \alpha, \beta$ such that for $\xx\in\I+\VV$,
	\begin{align}\begin{split}
	A(\xx) & =p\log\det(\xx)+\alpha, \\
	B(\xx) & =p\log\det(\xx)+\beta, \\
	C(\xx) & =p\log\det(\xx)+\alpha+\beta. \\
	\end{split}\end{align}
\end{lemma}
\begin{proof}
We {will first simplify the problem to solving a functional equation with one unknown function}: setting ${\yy}=2\I$ and ${\xx}=2\I$ successively, we get
	$$A(\xx)=C(2\xx)-B(2\I),\qquad B(\xx)=C(2\xx)-A(2\I)$$
	{for $\xx\in(\I+\VV)$.}
	Thus, using $A(2\I)+B(2\I)=C(4\I)$, we get
	$$C(2\xx)+C(2\yy)-C(4\I)=C(\xx\circ\yy).$$
	With $(\ss,\ttt)\in(\I+\VV)\times(\I+\VV)$ setting $\xx=2 \ss$ and $\yy=2\ttt$ above,
	we get
	$$C(4\ss) + C(4\ttt)-C(4\I)=C(4\, \ss\circ \ttt).$$
	Hence, function $f\colon (\I+\VV)\to\R$ defined by $f(\xx):=C(4\xx)-C(4\I)$ satisfies 
	\begin{align}\label{falone}
	f(\xx)+f(\yy)=f(\xx\circ\yy),\qquad (\xx,\yy)\in(\I+\VV)^2.
	\end{align}
Define an extension $\bar{f}\colon\VV\to\R$ of $f$ by
$$\widetilde{f}(\xx)=\begin{cases}
f(\xx), &\mbox{ if }\xx\in(\I+\VV), \\
f(\alpha_\xx \xx)-f(\alpha_\xx\I), &\mbox{ if }\xx\in\VV\setminus(\I+\VV),
\end{cases}
$$ 
where 
$$\alpha_{\xx}=\frac{2}{\min_{1\leq k\leq r} \{\lambda_k(\xx)\}}$$ and $\lambda_k(\xx)$ is $k$th eigenvalue of $\xx$. If $\xx\in\VV\setminus(\I+\VV)$, then there exists $k$ such that $\lambda_k(\xx)\in(0,1]$. Thus, in such case, $\alpha_{\xx}\geq 2$ and $\alpha_\xx \xx\in(\I+\VV)$.
It is easy to see that such extension satisfies 
\begin{align}\label{wf}
\widetilde{f}(\xx)+\widetilde{f}(\yy)=\widetilde{f}(\xx\circ\yy),\qquad (\xx,\yy)\in\VV^2.
\end{align}
{For example, assume that $\xx\in(\I+\VV)$, while $\yy$ and $\xx\circ\yy\in\VV\setminus(\I+\VV)$. Then, \eqref{wf} is equivalent to (after rearrangements)
\begin{align*}
f(\xx)+f(\alpha_\yy \yy)+f(\alpha_{\xx\circ\yy}\I) =f(\alpha_{\xx\circ\yy}\xx\circ\yy)+f(\alpha_\yy \I)
\end{align*}
and both sides equal $f(\alpha_\yy\alpha_{\xx\circ\yy}\xx\circ\yy)$ by \eqref{falone}. The other cases are treated accordingly (see \cite[Lemma 3.2]{BK13} for similar argument applied to additive equation).
}
By \cite[Theorem 3.4]{BK15}, we obtain $\widetilde{f}(\xx)=p\log\det(\xx)$ for some $p\in\R$, which ends the proof.
\end{proof}

\begin{lemma}\label{l1dim}
	Let $H\colon (0,1)\to\R$ and $G\colon(0,\infty)\to\R$ be continuous functions satisfying
	\begin{align}\label{1dim}
	H\left(\frac{x}{1+x}\right)+G(y)=H\left(\frac{1+x+y}{x+y}\frac{x}{1+x}\right)+G(x+y),\qquad(x,y)\in(0,\infty)^2.
	\end{align}
	Then, there exist $q, C_1, C_2\in\R$ such that 
	\begin{align*}
	H(x)&=q\log(1-x)+C_2, \quad x\in(0,1),\\
	G(x)&=q\log(x)+C_1, \,\,\,\,\quad\quad x\in(0,\infty).
	\end{align*}
\end{lemma}
\begin{remark}
	Above result {can be deduced} from {the proof of the main result in} \cite{KV11}. {However,} we decided to present {another proof which does not rely on arguments from \cite{KV11}.}
\end{remark}
\begin{proof}
	For {arbitrary} $s,t\in(0,\infty)$ and $\alpha>0$, set
	$$x_\alpha=\frac{1}{2}\left(\sqrt{4t+(1-\alpha st)^2}-1-\alpha st\right),\qquad y_\alpha=\alpha s t.$$
	For $\alpha$ sufficiently small, $x_\alpha>0$. Moreover,
	$$\frac{1+x_\alpha+y_\alpha}{x_\alpha+y_\alpha}\frac{x_\alpha}{1+x_\alpha}=1-\alpha s$$
	and 
	$$x_0(t):=\lim_{\alpha\to0} x_\alpha=\frac{1}{2}\left(\sqrt{4t+1}-1\right)\in (0,\infty).$$
Thus, passing to the limit in \eqref{1dim} {for $x=x_\alpha$ and $y=y_\alpha$} one eventually obtains for all $s,t>0$,
	$$
	f(t):=H\left(\frac{x_0(t)}{1+x_0(t)}\right)-G(x_0(t))=\lim_{\alpha\to0}\left\{H\left(1-\alpha s\right)-G(\alpha s t)\right\}.
	$$
	Inserting $(1,st)$ instead of $(s,t)$ above, we obtain
	$$
	f(st)=\lim_{\alpha\to0}\left\{H\left(1-\alpha \right)-G(\alpha s t)\right\}
	$$
	and after subtracting
	$$f(t)-f(st)=\lim_{\alpha\to0}\left\{ H\left(1-\alpha s\right)-H\left(1-\alpha \right)\right\},$$
	which is a pexiderized version of the Cauchy logarithmic functional equation.
	Since $f$ is continuous, this implies that there exist real constants $\beta, C$ such that
	$$f(t)=H\left(\frac{x_0(t)}{1+x_0(t)}\right)-G(x_0(t))=\beta\log t+C,$$
	which is equivalent to
	$$H(u)=G\left(\frac{u}{1-u}\right)+\beta\log\frac{u}{(1-u)^2}+C$$
	{for $u\in(0,1)$}.
	Using above in \eqref{1dim} and substituting $G(x)=-\beta\log x+\widetilde{G}(x)$
	we arrive at
	$$\widetilde{G}(x)+\widetilde{G}(y)=\widetilde{G}\left(\frac{x}{y}(1+x+y)\right)+\widetilde{G}(x+y),\qquad(x,y)\in(0,\infty)^2.$$
	Interchanging the roles of $x$ and $y$ and subtracting such obtained equation we get
	$$\widetilde{G}\left(\frac{y}{x}(1+x+y)\right)=\widetilde{G}\left(\frac{x}{y}(1+x+y)\right).$$
	Setting $x=t(st-1)/(s+t)$ and $y=s(st-1)/(s+t)$ above simplifies to
	$$\widetilde{G}\left(s^2\right)=\widetilde{G}\left(t^2\right)$$
	provided $s,t,st-1>0$. This implies that $\widetilde{G}\equiv C_2$ is a constant function and finally
	$$H(u)=-\beta\log\left(\frac{u}{1-u}\right)+\beta\log\frac{u}{(1-u)^2}+C+C_2.$$
\end{proof}

The following Theorem is the main technical result of the present paper. To prove it, 
we will use techniques developed in \cite{BK16a} and further used in \cite{BK16b, BK16c, BK17, AP17}. {Let
$$\DD=\{ \xx\in \VV\colon\I-\xx\in\VV\}.$$}
\begin{theorem}\label{mainfun}
	Let $f,g,k\colon \VV\to\R$ and $h\colon\mathcal{D}\to\R$ be continuous functions such that
	\begin{align}\label{maineq}
	f(\xx)+g(\yy)=k(\xx+\yy)+h\left((\I+(\xx+\yy)^{-1})\circ(\I+\xx^{-1})^{-1}\right),\qquad (\xx,\yy)\in\VV^2.
	\end{align}
	Then there exist $p,q\in\R$, $\cb\in\mathrm{Sym}(\R,r)$, $C_i\in\R$, $i=1,\ldots,4$, such that for $x\in\VV$ and $\ub\in\mathcal{D}$
	\begin{align}\label{sol}\begin{split}
	f(\xx)&=-\scalar{\cb,\xx}+p\log\det(\xx)-q\log\det(\I+\xx)+C_1,\\
	g(\xx)&=-\scalar{\cb,\xx}+(q-p)\log\det(\xx)+C_2,\\
	k(\xx)&=-\scalar{\cb,\xx}-p\log\det(\I+\xx)+q\log\det(\xx)+C_3,\\
	h(\ub)&=p\log\det(\ub)+(q-p)\log\det(\I-\ub)+C_4,
	\end{split}\end{align}
	with $C_1+C_2=C_3+C_4$.
\end{theorem}
\begin{proof}
	First nontrivial observation to be made is that \eqref{sol} solves \eqref{maineq}. This was proved in \cite{Kou12} and follows directly by Proposition \ref{ProK1}.
	
	The proof is divided into five Steps. In Steps 1 and 2 we simplify the problem so that we have two unknown functions instead of four as in original problem. Steps 3 and 4 are preparatory for the Step 5, in which we find the general form of function $h$. 
	\begin{description}
		\item[Step 1] 
		{For arbitrary $\ss, \vb\in\VV$ and $\alpha>0$ such that $\I-\alpha s\in\VV$, set in \eqref{maineq}}
		$$\xx=({\I}-\alpha \ss)^{-1}\cdot\alpha \ss,\qquad \yy=\vb-\xx.$$ 
		For $\alpha$ sufficiently small {(the upper bound for $\alpha$ depends on $\ss$)}, $\xx$ and $\yy$ belong to $\VV$.
		Moreover, $(\I+\xx^{-1})^{-1}=\alpha \ss$.
		Since $\xx\to\textbf{0}$ as $\alpha\to0$, we have $\yy\to\vb\in\VV$. By assumption, $g$ is continuous on $\VV$, 
		thus, passing in \eqref{maineq} to the limit as $\alpha\to0$ we obtain
		\begin{align}\label{gk}
		g(\vb)-k(\vb)=\lim_{\alpha\to0}\left\{h(\alpha (\I+\vb^{-1})\circ\ss)-f((\I-\alpha \ss)^{-1}\cdot \alpha \ss)\right\}.
		\end{align}
		Putting $\ss=\I$ we see that the limit
		$$C(\xx):=\lim_{\alpha\to0}\left\{h(\alpha \xx)-f\left(\frac{\alpha}{1-\alpha}\I\right)\right\}$$ exists if $\xx\in\I+\VV$.
		For $\ss\in\I+\VV$ and $\vb\in\VV$, we have $(\I+\vb^{-1})\circ\ss\in\I+\VV$ and thus the right hand side of \eqref{gk} equals
		\begin{align}\label{ST}
		C( (\I+\vb^{-1})\circ\ss)-\lim_{\alpha\to 0}\left\{f\left((1-\alpha \ss)^{-1} \cdot\alpha \ss\right)- f\left(\frac{\alpha}{1-\alpha}\I\right)\right\}.
		\end{align}
		Hence, \eqref{gk} with $g(\xx)-k(\xx)=:A(I+\xx^{-1})$ gives us 
		$$A(\I+\vb^{-1})+B({\ss})=C((\I+\vb^{-1})\circ\ss),\qquad (\vb,\ss)\in\VV\times(\I+\VV),$$
		where $B({\ss})$ denotes the second term in \eqref{ST}.
		Lemma \ref{PEX} then implies that there exist $p,\gamma_1\in\R$ such that
		$$k(\vb)=g(\vb)-p\log\det\left(\I+\vb^{-1}\right)+\gamma_1,\qquad \vb\in\VV.$$
		Using above in \eqref{maineq} and substituting 
		\begin{align*}
		f(\xx)&=p\log\det(\xx)+f_1(\xx)+\gamma_1,\\
		g(\xx)&=-p\log\det(\xx)+g_1(\xx),\\
		h(\ub)&=p\log\det(\ub)-p\log\det(\I-\ub)+h_1(\ub),
		\end{align*}
		we arrive at a functional equation with three unknown functions
		\begin{align}\label{maineq1}
		f_1(\xx)+g_1(\yy)=h_1((\I+(\xx+\yy)^{-1})\circ(\I+\xx^{-1})^{-1})+g_1(\xx+\yy),\qquad (\xx,\yy)\in\VV^2.
		\end{align}
		
		\item[Step 2]
		Set $\yy=\alpha\ttt\in\VV$ in \eqref{maineq1} and pass to the limit as $\alpha\to\infty$. Then, by continuity of $h$ (and so of $h_1$) we have
		\begin{align}\label{eqlin}
		f_1(\xx)-h_1((\I+\xx^{-1})^{-1})=\lim_{\alpha\to\infty}\left\{g_1(\xx+\alpha \ttt)-g_1(\alpha\ttt)\right\}
		\end{align}
		for any $\xx,\ttt\in\VV$. 
		Writing
		$$
		g_1(\xx+\yy+\alpha\xx)-g_1(\alpha\xx)=g_1(\yy+(\alpha+1)\xx)-g_1((\alpha+1)\xx)+g_1(\xx+\alpha\xx)-g_1(\alpha \xx)
		$$
		and passing to the limit as $\alpha\to\infty$, by \eqref{eqlin}, we obtain
		$$L(\xx+\yy)=L(\yy)+L(\xx),\qquad (\xx,\yy)\in\VV\times\VV,$$
		where $L(\xx)=f_1(\xx)-h_1((\I+\xx^{-1})^{-1})$. Thus, ({e.g.} by \cite[Lemma 3.2]{BK13}) we get
		$$f_1(\xx)=h_1((\I+\xx^{-1})^{-1})-\scalar{\cb,\xx},\qquad \xx\in\VV$$
		for some $\cb\in\mathrm{Sym}(r,\R)$. Inserting it into \eqref{maineq1} along with substitution
		$$g_1(\xx)=-\scalar{\cb,\xx}+g_2(\xx),\qquad\xx\in\VV,$$ 
		we arrive at
		\begin{align}\label{maineq2}
		h_1((\I+\xx^{-1})^{-1})+g_2(\yy)=h_1((\I+(\xx+\yy)^{-1})\circ(\I+\xx^{-1})^{-1})+g_2(\xx+\yy),\qquad (\xx,\yy)\in\VV^2.
		\end{align}
		
		\item[Step 3]
		Setting $\xx=x\I$ and $\yy=y\I$, \eqref{maineq2} boils down to a functional equation with scalar arguments
		$$
		h_1\left(\frac{x}{1+x}\I\right)+g_2(y\I)=h_1\left(\frac{1+x+y}{x+y}\frac{x}{1+x}\I\right)+g_2((x+y)\I),\qquad (x,y)\in(0,\infty)^2.
		$$
		Since functions $x\mapsto h_1(x\I)$ and $x\mapsto g_2(x\I)$ are continuous, by Lemma \ref{l1dim} we conclude that there exists a real constant $C_4$ such that $h_1(x \I)\to C_4$ as $x\to0$. 
		Now, observe that setting $\xx=\alpha \ss$ and $\yy=(\I+\alpha \ss)\cdot\ss\cdot(\beta\I-\ss)^{-1}$ we have ($\xx$ and $\yy$ commute)
		$$(\I+(\xx+\yy)^{-1})\circ(\I+\xx^{-1})^{-1}=\frac{\alpha\beta}{\alpha\beta+1}\I.$$
		It is clear that for any $\ss\in\VV$, $\yy\in\VV$ for $\beta$ sufficiently large {(the lower bound for $\beta$ depends on $\ss$)}.
		Since $\xx\to 0$ and $\yy\to\ss\cdot(\beta\I-\ss)^{-1}$ if $\alpha\to 0$, by \eqref{maineq2} we obtain for any $\ss\in\VV$,
		\begin{align*}\lim_{\alpha\to0}h_1\left(\left(\I+(\alpha \ss)^{-1}\right)^{-1}\right)+g_2\left(\ss\cdot(\beta\I-\ss)^{-1}\right)
		=C_4+g_2\left(\ss\cdot(\beta\I-\ss)^{-1}\right),
		\end{align*}
		that is, 
		$$\lim_{\alpha\to0}h_1\left(\left(\I+(\alpha \ss)^{-1}\right)^{-1}\right)=C_4.$$
		
		\item[Step 4]
		Set $\xx=\alpha \ss$ and $\yy=\alpha \ttt$ in \eqref{maineq2}.
		Then, as $\alpha\to 0$,
		$$(\I+(\xx+\yy)^{-1})\circ(\I+\xx^{-1})^{-1}\to (\ss+\ttt)^{-1}\circ\ss.$$
		By Step 3, passing to the limit as $\alpha\to0$ in \eqref{maineq2}, we obtain
		\begin{align}\label{last}
		h_1((\ss+\ttt)^{-1}\circ\ss)=C_4+\lim_{\alpha\to0}\left\{g_2(\alpha \ttt)-g_2(\alpha(\ss+\ttt))\right\}.
		\end{align}
		Setting $\ttt=\I$, we see that the limit 
		\begin{align}D(\xx):=\label{lim}\lim_{\alpha\to 0}\left\{g_2(\alpha \xx)-g_2(\alpha\I)  \right\}
		\end{align}
		exists for $\xx\in\I+\VV$ and equals $h_1({\I-}\xx^{-1})-C_4$. This, in turn, implies that the limit \eqref{lim} exists for any $\xx\in\VV$. Indeed, take $\ss\in\I+\VV$ in \eqref{last} to conclude that
		$$
		\lim_{\alpha\to0}\left\{g_2(\alpha \ttt)-g_2(\alpha(\ss+\ttt))\right\} + \lim_{\alpha\to0}\left\{g_2(\alpha(\ss+\ttt))-g_2(\alpha\I)\right\} = 
		\lim_{\alpha\to0}\left\{g_2(\alpha\ttt)-g_2(\alpha\I)\right\} 
		$$
		exists for any $\ttt\in\VV$, since the limits on the left hand side exist.
		Thus, we finally obtain for $\ss,\ttt\in\VV$,
		\begin{align*}
		h_1((\ss+\ttt)^{-1}\circ\ss)-C_4& =\lim_{\alpha\to0}\left\{g_2(\alpha \ttt)-g_2(\alpha\I)\right\}-\lim_{\alpha\to0}\left\{g_2(\alpha(\ss+\ttt))-g_2(\alpha\I)\right\} \\
		& = D(\ttt)-D(\ss+\ttt),
		\end{align*}
		which on the one hand is the multiplicative functional equation on restricted domain and, on the other, this is simplified Olkin-Baker equation on $\VV$ considered in \cite{BK16a}. Thus, by \cite[Theorem 3.6]{BK16a} with $a\equiv0$, $b=D$, $c=D$, $d=h_1-C_4$ and ${\mathfrak{g}}(\xx)\yy=\xx^{-1}\circ\yy$, we obtain in particular that for some $q\in\R$,
		\begin{align}\label{h1}
		h_1(\ub)=q\log\det(\I-\ub)+C_4,\qquad\ub\in\mathcal{D}.
		\end{align}
		\item[Step 5] 
		Use \eqref{h1} in \eqref{maineq2} and substitute
		\begin{align*}
		g_2(\xx)&=q\log\det(\xx)+g_3(\xx), \qquad \xx\in\VV.
		\end{align*}
		Then, \eqref{maineq2} simplifies to
		$$g_3(\yy)=g_3(\xx+\yy),\qquad (\xx,\yy)\in\VV^2,$$
		which means that $g_3\equiv const=:C_2$.
		
		It is easy to see that \eqref{sol} holds with $\gamma_1=C_1-C_4=C_3-C_2$, which ends the proof.
\end{description}
\end{proof}

\section{Probability laws on $\VV$}
We will consider absolutely continuous laws on $\VV$, which will be characterized in the next Section. Let $d\xx$ denote the Lebesgue measure on {$\left(\mathrm{Sym}(r,\R),\scalar{\cdot,\cdot}\right)$}. 
For $p>(r-1)/2$ and $\sigma\in\VV$, we define the Wishart distribution $\gamma(p,\sigma)$ by its density
\begin{align*}
\gamma(p,\sigma)(\mathrm{d}\xx)=c_{p,\sigma} \det(\xx)^{p-(r+1)/2}\exp(-{\scalar{\sigma,\xx}})\II_{\VV}(\xx)\mathrm{d}\xx,
\end{align*}
where $c_{p,\sigma}$ is a norming constant. 
Beta distribution $Beta(p,q)$ is defined for $p>(r-1)/2$ and $q>(r-1)/2$ by density
\begin{align*}
Beta(p,q)(\mathrm{d}\xx)=c_{p,q} \det(\xx)^{p-(r+1)/2}\det(\I-\xx)^{q-(r+1)/2}\II_{\DD}(\xx)\mathrm{d}\xx,
\end{align*}
where $\DD=\{\xx\in\VV\colon \I-\xx\in\VV\}$ {and $c_{p,q}$ is a norming constant}.

The matrix Kummer distribution was first introduced in \cite{GCN01}; for $a>(r-1)/2$, $b\in\R$ and $\sigma\in\VV$, 
\begin{align*}
K(a,b,\sigma)(\mathrm{d}\xx)=c_{a,b,\sigma} \det(\xx)^{a-(r+1)/2}\det(\I+\xx)^{-b}\exp(-{\scalar{\sigma,\xx}})\II_{\VV}(\xx)\mathrm{d}\xx.
\end{align*}
(Note the misprint in the power of $\det(\I+\xx)$ in \cite[(2.3)]{Kou12})

By $\mu\otimes\nu$ we will denote the product measure of $\mu$ and $\nu$.

\section{The Matsumoto--Yor property of Kummer and Wishart random matrices}
The following Proposition was proved in \cite{Kou12}. It is used in the proof of the next Theorem, which is the main result of \cite{Kou12}:
\begin{proposition}\label{ProK}
A mapping $\psi\colon \VV\times\VV\to\DD\times\VV$ 
defined by 
\begin{align}\label{psi}
\psi(\xx,\yy)=\left(\left(\I+(\xx+\yy)^{-1}\right)\circ\left(\I+\xx^{-1}\right)^{-1}, \xx+\yy\right)
\end{align}
is a diffeomorphism and its Jacobian equals
\begin{align}\label{Jac}
J(\xx,\yy)=\frac{\det(\I+(\xx+\yy)^{-1})^{(r+1)/2}}{\det(\I+\xx)^{r+1}}.
\end{align}
\end{proposition}

\begin{theorem}\label{property}
	Let 
	$$(X,Y)\sim K(a,b,\sigma)\otimes\gamma(b-a,\sigma)$$
	and define 
	\begin{align}\label{UVXY}
	(U,V)=\psi(X,Y),
	\end{align}
	where $\psi$ is given by \eqref{psi}.
	Then 
	$$(U,V)\sim Beta(a,b-a)\otimes K(b,a,\sigma).
	$$
\end{theorem}

The independence property established in \cite{Kou12} has been proved for the case $r = 1$ in \cite{KV11}, where the converse has been proved too, thus providing a characterization of Kummer and gamma laws under the assumption of existence of smooth densities of $X$ and $Y$. 
At the end of \cite{Kou12}, the author writes that ``It is highly likely, although not easy to prove, that this characterization holds also in the case of matrices.'' 
We confirm his belief in the following Theorem, which shows that {matrix} Kummer and Wishart laws are the only laws having property given in Theorem \ref{property}.
\begin{theorem}\label{mainprob}
Let $X$ and $Y$ be two independent random matrices valued in $\VV$ with {continuous densities, which are strictly positive on $\VV$}. The random matrices $U$ and $V$ defined {in \eqref{UVXY}} are independent if and only if $X$ follows the matrix Kummer distribution $K(a, b, \sigma)$ and $Y$ the Wishart distribution $\gamma(b - a, \sigma)$ for some $a, b, \sigma$ with $a>\frac{r-1}{2}$, $b-a>\frac{r-1}{2}$ and $\sigma\in\mathcal{S}_r^+$.
\end{theorem}
\begin{proof}
The following identity holds almost everywhere with respect to Lebesgue measure
\begin{align}\label{dens}
f_{(X,Y)}(\xx,\yy)=f_{(U,V)}(\psi(\xx,\yy))J(\xx,\yy),
\end{align}
where $\psi\colon \VV\times\VV\to \DD\times\VV$ is the bijection given by \eqref{psi} and $J$ is the Jacobian of $\psi$. 
By independence, we have $f_{(X,Y)}=f_X f_Y$ and similarly $f_{(U,V)}=f_U f_V$. Since the densities of $X$ and $Y$ are assumed to be continuous, the above equation
holds for every $\xx,\yy\in\VV$.

After taking logarithms (it is permitted, since $f_X$ and $f_Y$ are assumed to be strictly positive on $\VV$) and using \eqref{Jac}, we arrive at \eqref{maineq} with
\begin{align*}
f(\xx)&=\log f_X(\xx)+(r+1)\log\det(\I+\xx), \\
g(\xx)&=\log f_Y(\xx),\\
k(\xx)&=\log f_V(\xx)+\frac{r+1}{2}\log\det(\I+\xx^{-1}),\\
h(\ub)&=\log f_U(\ub), 
\end{align*}
for $\xx\in\VV$ and $\ub\in\mathcal{D}$.
Thus, {by Theorem \ref{mainfun} we conclude that}
\begin{align*}
f_X(\xx)&=\exp\{f(\xx)\}\det(\I+\xx)^{-(r+1)}=e^{C_1}\det(\xx)^p \det(\I+\xx)^{-(r+1+q)}e^{-\scalar{\cb,\xx}},\\
f_Y(\xx)&=\exp\{g(\xx)\}=e^{C_2}\det(\xx)^{q-p}e^{-\scalar{\cb,\xx}},
\end{align*}
which are integrable on $\VV$ if and only if $\cb\in\VV$, $p>-1$ and $q-p>-1$ that is,
$$(X,Y)\sim K(a,b,\sigma)\otimes\gamma(b-a,\sigma)$$
with $a=p+(r+1)/2$ and $b=q+r+1$.
\end{proof}

\subsection*{Acknowledgment} This research was partially supported by NCN Grant No. 2016/21/B/ST1/00005. The author thanks A. Piliszek for helpful discussions.

\bibliographystyle{plainnat}


\def\polhk#1{\setbox0=\hbox{#1}{\ooalign{\hidewidth
  \lower1.5ex\hbox{`}\hidewidth\crcr\unhbox0}}}

\end{document}